\documentclass[final]{siamltex}

\usepackage{graphicx}
\usepackage{color}
\usepackage{bm}
\usepackage{amsfonts}
\usepackage{amsmath}

\newcommand{\tr}{^{\sf T}}
\newcommand{\m}[1]{{\bf{#1}}}
\newcommand{\g}[1]{\mbox{\boldmath $#1$}}
\newcommand{\C}[1]{{\cal {#1}}}

\input epsf

\title{
Extension of Switch Point Algorithm to Boundary-Value Problems
\thanks{July 1, 2023.
The author gratefully acknowledges support by the National
Science Foundation under grants 1819002 and 2031213, and by
Office of Naval Research under grant N00014-22-1-2397.}
}

\author{
    William W. Hager\thanks{{\tt hager@ufl.edu},
        http://people.clas.ufl.edu/hager/,
        PO Box 118105,
        Department of Mathematics,
        University of Florida, Gainesville, FL 32611-8105.}
}

\begin{document}

\maketitle

\begin{abstract}
In an earlier paper (https://doi.org/10.1137/21M1393315),
the Switch Point Algorithm was developed for solving optimal
control problems whose solutions are either singular or bang-bang
or both singular and bang-bang,
and which possess a finite number of jump discontinuities in an
optimal control at the points in time where the solution structure changes.
The class of control problems that were considered had a given initial
condition, but no terminal constraint.
The theory is now extended to include problems with both initial and terminal
constraints, a structure that often arises in boundary-value problems.
Substantial changes to the theory are needed to handle this more general
setting.
Nonetheless,
the derivative of the cost with respect to a switch point is again
the jump in the Hamiltonian at the switch point.
\end{abstract}

\begin{keywords}
Switch Point Algorithm, Singular Control, Bang-Bang Control,
Boundary-value Problems
\end{keywords}

\begin{AMS}
49M25, 49M37, 65K05, 90C30
\end{AMS}

\pagestyle{myheadings}
\thispagestyle{plain}
\markboth{W. W. HAGER}{SWITCH POINT ALGORITHM FOR BVP}

\section{Introduction}
%
%
An earlier paper \cite{AghaeeHager21} develops the Switch Point Algorithm for
initial-value problems with bang-bang or singular solutions.
This paper extends the algorithm to problems with terminal constraints.
More precisely, we consider fixed terminal time control problem of the form
\begin{equation}\label{CP}
\begin{array}{ll}
\min C(\m{x}(T)) \quad \mbox{subject to} & \m{\dot{x}}(t) =
\m{f}(\m{x}(t), \m{u}(t)), \quad \m{u}(t) \in \C{U}(t), \\[.05in]
~& \m{x}_I (0) = \m{b}_I, \quad \m{x}_E (T) = \m{b}_E,
\end{array}
\end{equation}
where $\m{x} : [0, T] \rightarrow \mathbb{R}^n$ is absolutely continuous,
$\m{u} : [0, T] \rightarrow \mathbb{R}^m$ is essentially bounded,
$C : \mathbb{R}^n \rightarrow \mathbb{R}$,
$\m{f} : \mathbb{R}^n \times \mathbb{R}^m \rightarrow$ $\mathbb{R}^n$,
$\C{U}(t)$ is a closed and bounded set for each $t \in [0, T]$, $I$ and $E$
are subsets of $\{1, 2, \ldots, n\}$, and $\m{x}_I$ denotes the
subvector of $\m{x}$ associated with indices $i \in I$.
The vectors $\m{b}_I$ and $\m{b}_E$ are given initial and terminal values
for the state.
It is assumed that $|I| + |E| = n$, where $|S|$
denotes the number of elements in a set $S$, and
the dynamics $\m{f}$ and the objective $C$ are continuously differentiable.
Here and throughout the paper, differential equations should hold almost
everywhere on $[0, T]$.
Problems of this form arise in boundary-value problems such as the fish
harvesting problem in \cite{Neubert03}, which is also studied in the PhD thesis
\cite{Atkins21} of Summer Atkins.

With the notation given above, the paper \cite{AghaeeHager21} considered
an initial value problem where $|I| = n$ and $|E| = 0$.
In this special case, any $\m{u}$ satisfying the control constraint is
feasible, and the associated state is the solution to an initial value problem.
When $|E| > 0$, components of the
initial state corresponding to $i \in I^c$, the complement of $I$,
are unknown.
The nonspecified components of the initial state along with the control
$\m{u}$ must be chosen to satisfy the boundary condition
$\m{x}_E (T) = \m{b}_E$.
Due to the terminal constraint, the theory developed in \cite{AghaeeHager21}
is no longer applicable.

The costate associated with (\ref{CP}) satisfies the linear
differential equation
\begin{equation}\label{p}
\m{\dot{p}}(t) =
-\m{p}(t)\nabla_x \m{f}(\m{x}(t), \m{u}(t)),
\quad \m{p}_J (0) = \m{0},
\quad \m{p}_F(T) = \nabla_F C(\m{x}(T)) ,
\end{equation}
where $J$ and $F$ denote the complements of $I$ and $E$ respectively,
$\m{p} : [0, T] \rightarrow \mathbb{R}^n$ is a row vector,
the objective gradient $\nabla_F C$ is a row vector whose $i$-th component
is the partial derivative of $C$ with respect to $x_i$, $i \in F$, and
$\nabla_x \m{f}$ denotes the Jacobian of the dynamics with respect to $\m{x}$.
Due to the terminal constraint $\m{x}_E (T) = \m{b}_E$, the objective is only
a function of $\m{x}_F (T)$.
Under the assumptions of the Pontryagin minimum principle,
a local minimizer of (\ref{CP}) and the associated costate have the
property that
\begin{equation}\label{PMP}
H (\m{x}(t), \m{u}(t), \m{p}(t)) =
\inf \{ H (\m{x}(t), \m{v}, \m{p}(t)) : \m{v} \in \C{U}(t) \}
\end{equation}
for almost every $t \in [0,T]$, where
$H(\m{x}, \m{u}, \m{p}) = \m{p}\m{f}(\m{x}, \m{u})$ is the Hamiltonian.

When the Hamiltonian is linear in the control and the
feasible control set has the form
\[
\C{U}(t) = \{ \m{v}\in\mathbb{R}^m :
\g{\alpha}(t) \le \m{v} \le \g{\beta}(t) \},
\]
where $\g{\alpha}$ and $\g{\beta} : [0,T] \rightarrow \mathbb{R}^m$,
it is often possible to decompose $[0, T]$ into a finite number of
disjoint subintervals $(s_i, s_{i+1})$, where
$0 = s_0 < s_1 < \ldots < s_N = T$, and on each subinterval,
each component of an optimal control is either singular or bang-bang.
Moreover, by singular control theory \cite{Schattler12}, it is often possible to
express the control in feedback form as
$\m{u}(t) =$ $\g{\phi}_i(\m{x}(t),t)$ for all $t \in (s_i, s_{i+1})$ for some
function $\g{\phi}_i$ defined on a larger interval containing $(s_i, s_{i+1})$.
In the Switch Point Algorithm,
the original control problem is solved by optimizing
over the choice of the $s_i$, $0 < i < N$.
In other words, if $\m{F}_i(\m{x}, t) := \m{f}(\m{x}, \g{\phi}_i(\m{x},t))$ and
$\m{F}(\m{x},t) := \m{F}_i(\m{x},t)$ for all $t \in (s_i, s_{i+1})$,
$0 \le i < N$, then (\ref{CP}) is replaced by the problem
\begin{equation}\label{SP}
\min_{\m{s}} C(\m{x}(T)) \quad \mbox{subject to} \quad \m{\dot{x}}(t) =
\m{F}(\m{x}(t),t ), \quad \m{x}_I (0) = \m{b}_I, \quad \m{x}_E (T) = \m{b}_E.
\end{equation}

In order to solve (\ref{SP}) efficiently,
we develop an algorithm for computing the derivative of the objective
with respect to a switch point.
This formula allows us to utilize gradient, conjugate gradient, and
quasi-Newton methods in the solution process.
Let $C(\m{s})$ denote the objective in (\ref{SP}) parameterized by the
switch points $s_i$, $0 < i < N$.
Under a smoothness assumption for each $\m{F}_i$ and invertibility
assumptions for submatrices of related fundamental matrices,
we obtain the following formula:
\begin{equation} \label{DC}
\frac{\partial C}{\partial s_i} (\m{s}) =
H_{i-1} (\m{x}(s_i), \m{p}(s_i), s_i)
- H_{i}(\m{x}(s_i), \m{p}(s_i), s_i), \quad 0 < i < N,
\end{equation}
where $H_i (\m{x}, \m{p}, t) = \m{p}\m{F}_i(\m{x}, t)$,
and the row vector $\m{p}: [0, T] \rightarrow \mathbb{R}^n$
is the solution to the linear differential equation
\begin{equation}\label{pdj}
\m{\dot{p}}(t) = - \m{p}(t) \nabla_x \m{F}(\m{x}(t), t),
\quad t \in [0, T],
\quad
\m{p}_F(T) = \nabla_F C(\m{x}(T)), \quad \m{p}_J(0) = \m{0}.
\end{equation}
This matches the formula given in \cite[Thm. 2.4]{AghaeeHager21} in the
case $|E| = 0$.
Summer Atkins in her thesis \cite{Atkins21} also obtains this formula
in the special case of the fish harvesting problem.
Since $\m{F}$ could jump at $s_i$,
the existence of the Jacobian in (\ref{pdj}) is generally restricted to
the open intervals $(s_i, s_{i+1})$, and the differential equation only
needs to hold almost everywhere.

See the earlier paper \cite{AghaeeHager21} for a detailed survey of literature
concerning bang-bang and singular control problems, which includes the
papers \cite{
Aly1978,
Anderson1972,
Biegler20,
Biegler19b,
Betts2010,
Bryson75,
Biegler16,
Biegler19a,
bang_bang, Jacobson70, Maurer1976, Maurer05,
VossenThesis, Vossen2010}.
In more recent work \cite{PagerRao2022}, the authors develop a method for
solving bang-bang and singular optimal control problem using adaptive
Legendre–Gauss–Radau collocation
\cite{DarbyHagerRao11, DarbyHagerRao10, HagerHouRao19, LiuHagerRao15,
LiuHagerRao18, PattersonHagerRao15}
in which the structure of the solution is
first determined, and a regularization technique is used
in the singular regions, while
the switch points are treated as free parameters in the optimization.
The gradient methods that might be used in conjunction with the derivatives
provided in the current paper do not require regularization, however,
as discussed in Section~\ref{numerical}, a good
starting guess for the switch points is needed to ensure convergence.

The paper is organized as follows.
Section~\ref{existence} provides an existence result for a system of
nonlinear equations.
This key result is the basis for a stability analysis of the boundary-value
problem associated with (\ref{CP}).
In Section~\ref{SBC}, stability with respect the terminal boundary constraint
is analyzed, while Section~\ref{SSP} analyzes stability with respect
to a switch point.
In Section~\ref{DSP}, the results of the previous sections are combined to
obtain the derivative formula~(\ref{DC}).
Section~\ref{case2} discusses problems where a singular control depends
on both state and costate.
Finally, Section~\ref{numerical} explores numerical issues.

{\bf Notation and Terminology.}
Throughout the paper, $\| \cdot \|$ is any norm on $\mathbb{R}^n$.
The ball with center $\m{c} \in \mathbb{R}^n$ and radius $\rho$ is denoted
$\C{B}_\rho (\m{c}) = \{ \m{x} \in \mathbb{R}^n : \|\m{x} - \m{c}\| \le \rho\}$.
The expression $\C{O}(\g{\theta})$ denotes a quantity whose norm is bounded
by $c \|\g{\theta}\|$, with $c$ is a constant
that is independent of $\g{\theta}$.
The Jacobian of $\m{f}(\m{x}, \m{u})$ with respect to $\m{x}$ is denoted
$\nabla_x \m{f}(\m{x}, \m{u})$; its $(i, j)$ element is
$\partial f_i (\m{x}, \m{u})/ \partial x_j$.
For a real-valued function such as $C$, the gradient
$\nabla C(\m{x})$ is a row vector, while
$\nabla_F C(\m{x})$ is a row vector whose $i$-th component, $i \in F$,
is the partial derivative of $C$ with respect to $x_i$.
For a vector $\m{x} \in \mathbb{R}^n$ and a set
$I \subset \{1, 2, \ldots, n\}$, $\m{x}_I$ is the subvector consisting of
elements  $x_i$, $i \in I$.
If $\m{A} \in \mathbb{R}^{m \times n}$ is a matrix, and $R$ and $C$ are
subsets of the row and column numbers respectively, then
$\m{A}_{RC}$ is the submatrix corresponding to rows in $R$ and columns in $C$.
All vectors in the paper are column vectors except for the costate $\m{p}$
which is a row vector.

\section{An Existence Result}
\label{existence}
In order to derive the formula (\ref{DC}) for the derivative of the objective
with respect to a switch point, we first need to analyze the stability of the
boundary-value problem in (\ref{CP}).
This analysis is done using the proposition stated below.
The proposition is a very special case of a general theorem
given in \cite[Thm. 2.1]{Hager02b}.
The general result, formulated in a Banach space with set-valued maps,
has broad application in the convergence analysis of numerical algorithms,
as seen in papers such as
\cite{HagerDontchevPooreYang95,DontchevHagerVeliov00,
Hager90, Hager2000RungeKuttaMI}.
The special case stated here is for finite dimensional point-to-point maps
which is sufficient for handling the analysis of (\ref{CP}).
This result is closely related to Newton's method, a favorite topic of
Asen L. Dontchev, whom we remember in this volume.

\begin{proposition}\label{asen}
Suppose that $\m{g} : \mathbb{R}^n \rightarrow \mathbb{R}^n$ is continuously
differentiable in $\C{B}_r(\m{0})$ for some $r > 0$,
and define $\delta =$ $\|\m{g}(\m{0})\|$.
Let $\C{L} \in \mathbb{R}^{n \times n}$ be an invertible matrix with
$\gamma := \|\C{L}^{-1}\|$ and with the
property that for some $\epsilon > 0$,
\begin{equation}\label{er}
\|\nabla g (\g{\theta}) - \C{L}\| \le \epsilon \; \mbox{ for all }
\g{\theta} \in \C{B}_r (\m{0}).
\end{equation}
If $\epsilon \gamma < 1$ and $\delta \le r(1-\gamma \epsilon)/\gamma$,
then there exists a unique $\g{\theta} \in \C{B}_r(\m{0})$ such that
$\m{g}(\g{\theta}) = \m{0}$.
Moreover, we have the bound
\begin{equation}\label{errorbound}
\|\g{\theta}\| \le \frac{\delta \gamma}{1 - \epsilon \gamma} .
\end{equation}
\end{proposition}

\section{Stability with Respect to Terminal Constraint}
\label{SBC}
In analyzing the differentiability of the objective in (\ref{SP}) with
respect to a switch point, there no loss in generality in focusing on the
case $N = 2$, where there is a single switch point $s \in (0, T)$ and
the dynamics switches from $\m{F}_0$ to $\m{F}_1$ at $t = s$:
\[
\m{F}(\m{x}, t) = \m{F}_0 (\m{x}, t) \; \mbox{for all} \; t \in [0, s)
\quad \mbox{and} \quad 
\m{F}(\m{x}, t) = \m{F}_1 (\m{x}, t) \; \mbox{for all} \; t \in (s, T] .
\]
It is assumed that there exists a feasible,
absolutely continuous state $\m{x}$ which satisfies the constraints
of (\ref{SP}).
That is, $\m{x}$ satisfies
\begin{equation}\label{s}
\m{\dot{x}}(t) =
\m{F}(\m{x}(t),t ), \quad \m{x}_I (0) = \m{b}_I, \quad \m{x}_E (T) = \m{b}_E.
\end{equation}
Throughout the paper, $\m{x}$ denotes a solution to this problem.
In this section, we focus on the following question:
If the endpoint constraint $\m{b}_E$ in (\ref{s}) is changed to
$\m{b}_E + \g{\pi}$, does there exist a solution $\m{x}^{\pi}$
to the perturbed problem
\begin{equation}\label{P}
\m{\dot{x}}(t) =
\m{F}(\m{x}(t),t ), \quad \m{x}_I (0) = \m{b}_I,
\quad \m{x}_E (T) = \m{b}_E + \g{\pi},
\end{equation}
and is the solution change bounded in terms of $\|\g{\pi}\|$?
The following assumption is used in this analysis.

{\it Dynamics Smoothness.}
For $\rho > 0$, define the tubes
\begin{eqnarray*}
\C{T}_0 &=& \{ (\g{\chi}, t) : t \in [0, s+\rho]
\mbox{ and } \g{\chi} \in \C{B}_\rho(\m{x}(t)) \}, \\
\C{T}_1 &=& \{ (\g{\chi}, t) : t \in [s-\rho, T]
\mbox{ and } \g{\chi} \in \C{B}_\rho(\m{x}(t)) \}.
\end{eqnarray*}
It is assumed that on $\C{T}_j$, $j = 0$ or 1,
$\m{F}_j$ is continuously differentiable,
while $\m{F}_j(\g{\chi}, t)$ is Lipschitz continuously
differentiable in $\g{\chi}$, uniformly in $t$, with Lipschitz constant $L$.

Let us define $\g{\theta}^* = \m{x}_J(0)$, and let us consider the initial-value
problem
\begin{equation}\label{IVP}
\m{\dot{y}}(t) =
\m{F}(\m{y}(t),t ), \quad \m{y}_I (0) = \m{b}_I,
\quad \m{y}_J (0) = \g{\theta}^* + \g{\theta}.
\end{equation}
For $\g{\theta} = \m{0}$, $\m{y} = \m{x}$, the solution of (\ref{s}),
since $\m{y}_J(0) = \m{x}_J(0)$.
Under Dynamics Smoothness, it follows from
\cite[Cor. 2.3]{AghaeeHager21} that (\ref{IVP}) has a solution
$\m{y}_{\theta}$ when $\|\g{\theta} \|$ is sufficiently small, and we have
the bound
\begin{equation}\label{yx}
\|\m{y}_{\theta}(t) - \m{x}(t)\| =
\|\m{y}_{\theta}(t) - \m{y}_0(t)\| \le e^{Lt} \|\g{\theta} \|
\quad \mbox{for all } t \in [0, T].
\end{equation}
By the continuity of $\nabla_x \m{F}_j$ on $\C{T}_j$, for $j = 0$ or 1,
it follows that there is a constant $\beta$ such that
\begin{equation}\label{beta}
\|\nabla_x \m{F}(\g{\chi}, t)\| \le \beta \; \mbox{ for all } \;
t \in [0, T] \;\mbox{ and } \; \g{\chi} \in \C{B}_\rho (\m{x}(t)).
\end{equation}

A sharper estimate for the  difference $\m{y} - \m{x}$ is obtained from the
solution $\m{z}_{\theta}$ of the linearized problem
\begin{equation}\label{LDE}
\m{\dot{z}}(t) = \nabla_x F (\m{x}(t), t) \m{z}(t), \quad
\m{z}_I (0) = \m{0},
\quad \m{z}_J (0) = \g{\theta} .
\end{equation}
Since $\nabla_x \m{F}(\m{x}(t), t)$ is continuous on $[0, s)$ and on
$(s, T]$, the solution
to the linear differential equation (\ref{LDE}) has a bound
\begin{equation}\label{ztheta}
\m{z}_{\theta}(t) = O(\g{\theta}) \mbox{ for all } t \in [0, T].
\end{equation}
Define for all $t \in [0, T]$ and $\alpha \in [0, 1]$,
\begin{equation}\label{delta-def}
\g{\delta}(t) = \m{y}_{\theta}(t) - \m{x}(t) - \m{z}_{\theta}(t)
\quad \mbox{and} \quad
{\m{x}}(\alpha,t) = \m{x}(t) + \alpha(\m{y}_{\theta} (t) - \m{x}(t)).
\end{equation}
Differentiating $\g{\delta}$ and utilizing a Taylor expansion with
integral remainder term, we obtain for all $t \in [0, T]$, $t \ne s$,
$\g{\dot{\delta}}(t) =$
$\m{\dot{y}}_{\theta}(t) - \m{\dot{x}}(t) - \m{\dot{z}_{\theta}}(t) =$
\begin{eqnarray}
&& \m{F}(\m{y}_{\theta}(t), t) - \m{F} (\m{x}(t), t) -
\nabla_x \m{F} (\m{x}(t), t) \m{z}_{\theta}(t) \label{deltadot} = \\
&& \left( \int_0^1 \nabla_x \m{F} ({\m{x}}(\alpha,t), t) \; d\alpha
\right) (\m{y}_{\theta}(t) - \m{x}(t))
- \left( \int_0^1 \nabla_x \m{F} (\m{x}(t), t) \; d\alpha \right)
\m{z}_{\theta} (t) = \nonumber \\
&& \left( \int_0^1 [\nabla_x \m{F} ({\m{x}}(\alpha,t), t) -
\nabla_x \m{F} (\m{x}(t), t) ] \; d\alpha \right) \m{z}_{\theta}(t)
+ \left( \int_0^1 \nabla_x \m{F} ({\m{x}}(\alpha,t), t) \; d\alpha \right)
\g{\delta}(t) . \nonumber
\end{eqnarray}
Take $\g{\theta}$ in (\ref{yx})
small enough that $\m{y}_{\theta}(t)$ lies in the
tube around $\m{x}(t)$ where $\nabla_x \m{F}$ is Lipschitz continuous.
If $L$ is the Lipschitz constant for $\nabla_x \m{F}$, then we have
\begin{equation}\label{F-lip}
\|\nabla_x \m{F} ({\m{x}}(\alpha,t), t) - \nabla_x \m{F} (\m{x}(t), t)\|
\le \alpha L \|\m{y}_{\theta}(t) - \m{x}(t)\| = \C{O}(\g{\theta})
\end{equation}
by (\ref{yx}).
Take the norm of each side of (\ref{deltadot}).
On the right side of (\ref{deltadot}), the coefficient of $\m{z}_\theta$
is $\C{O}(\g{\theta})$ by (\ref{F-lip}), while $\m{z}_\theta$ is
$\C{O}(\g{\theta})$ by (\ref{ztheta}).
Since $\|\nabla_x \m{F} ({\m{x}}(\alpha,t), t)\| \le \beta$
for all $\alpha \in [0,1]$ and $t \in [0, T]$ by (\ref{beta}), the right side
of (\ref{deltadot}) has the bound
$\C{O}(\|\g{\theta}\|^2) + \beta \|\g{\delta}(t)\|$.
On the left side, exploit the fact from
\cite[Lem. 2.1]{AghaeeHager21}
that the derivative of a norm is bounded by the norm of the derivative
to obtain
\begin{equation}\label{dtheta}
\frac{d \|\g{\delta}(t)\|}{dt} \le
\|\g{\dot{\delta}}(t)\| \le
\C{O}(\|\g{\theta}\|^2) + \beta \|\g{\delta}(t)\|.
\end{equation}
By the initial conditions for $\m{y}_\theta$, $\m{x}$, and $\m{z}_\theta$
in (\ref{IVP}), (\ref{s}), and (\ref{LDE}) respectively,
$\g{\delta}(0) = \m{0}$.
This observation, together with (\ref{dtheta}) and
Gronwall's inequality yield
\begin{equation}\label{deltabound}
\|(\m{y}_\theta - \m{x}) - \m{z}_\theta\| =
\|\g{\delta}(t)\| = \C{O}(\|\g{\theta}\|^2).
\end{equation}
Thus $\m{z}_{\theta}$ provides an $\C{O}(\|\g{\theta}\|^2)$ approximation to
the difference $\m{y}_{\theta} - \m{x}$.

The linearized problem (\ref{LDE}) plays a fundamental role in the
stability analysis of (\ref{s}).
Finding a solution of the perturbed problem (\ref{P})
is equivalent to finding the  starting condition
$\g{\theta}$ in (\ref{IVP}) with the property that
$\m{y}_{\theta}(T) =$ $\m{b}_E + \g{\pi}$.
Since $\m{z}_{\theta}$ is a close approximation to $\m{y}_{\theta} - \m{x}$,
we could choose $\g{\theta}$ so that $\m{z}_{\theta}(T) = \g{\pi}$,
in which case
\[
\m{y}_{\theta}(T) = \m{x}(T) + \m{z}_{\theta}(T) + \C{O}(\|\g{\theta}\|^2) =
\m{b}_E + \g{\pi} + \C{O}(\|\g{\theta}\|^2).
\]
Therefore, for this choice of $\g{\theta}$,
the solution of (\ref{IVP}) satisfies
the perturbed boundary condition to within $\C{O}(\|\g{\theta}\|^2)$.

The fundamental matrix
$\g{\Phi}: [0, T] \rightarrow \mathbb{R}^{n \times n}$
associated with the linear system
$\m{\dot{z}}(t) = \nabla_x F (\m{x}(t), t) \m{z}(t)$
is the solution to the initial-value problem
\begin{equation}\label{Phi}
\g{\dot{\Phi}} (t) = \nabla_x F (\m{x}(t), t) \g{\Phi}(t), \quad
\g{\Phi}(0) = \m{I},
\end{equation}
where $\m{I}$ is the $n \times n$ identity matrix.
The solution $\m{z}$ of the linearized problem (\ref{LDE})
is equal to the fundamental matrix times the initial condition.
Due to the special choice of the initial condition in (\ref{LDE}),
the $\g{\theta}$ that yields
$\m{z}_E(T) = \g{\pi}$ is the solution to the linear system of equations
$\g{\Phi}_{EJ}(T)\g{\theta} = \g{\pi}$, where
$\g{\Phi}_{EJ}$ represents the submatrix of $\g{\Phi}$ associated with
columns $J$ and rows $E$.
If this square submatrix is invertible, then
$\g{\theta} =$ $\g{\Phi}_{EJ}(T)^{-1} \g{\pi}$.
With these insights, we have the following result:

\begin{lemma}\label{IC-stability}
Suppose that $\g{\Phi}_{EJ}(T)$ is invertible and let $\gamma =$
$\|\g{\Phi}_{EJ}^{-1}(T)\|$.
For $\g{\pi}$ in a neighborhood $\C{N}$ of the origin,
the perturbed boundary-value problem $(\ref{P})$
has a solution $\m{x}^{\pi}$ and
\begin{equation}\label{xpi-stable}
\|\m{x}_J^{\pi}(0) - \m{x}_J(0)\| =
\|\m{x}_J^{\pi}(0) - \g{\theta}^*\| \le c \|\g{\pi}\|
\mbox{ for all } \g{\pi} \in \C{N},
\end{equation}
where $c$ is a constant that approaches $\gamma$
as $\|\g{\pi}\|$ approaches $\m{0}$.
\end{lemma}
\begin{proof}
We apply Proposition~\ref{asen} with
$\C{L} =$ $\nabla \m{g}(\m{0})$, where
$\m{g}(\g{\theta}) =$
$\m{y}_{\theta E}(T) - \m{b}_E - \g{\pi}$ and $\m{y}_\theta$ is the
solution of (\ref{IVP}).
Both $\m{b}_E$ and $\g{\pi}$ are independent of $\g{\theta}$ so their
derivatives are $\m{0}$.
From the analysis in \cite[Chap.~1.6]{TaylorPDE},
the derivative of $\m{y}_{\theta E} (T)$ with respect to $\g{\theta}$,
evaluated at $\g{\theta} = \m{0}$ is
$\C{L} = \g{\Phi}_{EJ}(T)$.
Moreover, it follows from \cite[Chap.~1.6]{TaylorPDE} that
$\nabla \m{g}(\g{\theta})$ is continuously differentiable
at $\g{\theta} = \m{0}$.
Choose $\epsilon$ small enough that $\epsilon \gamma < 1$ and then
choose $r$ small enough that (\ref{er}) holds;
by continuity of the derivative of $\m{g}$ at $\g{\theta} = \m{0}$,
(\ref{er}) holds for $r$ sufficiently small.
Since $\m{g}(\m{0}) = \g{\pi}$, we have $\delta = \|\g{\pi}\|$.
Choose $\|\g{\pi}\|$ small enough that
$\delta \le r(1-\gamma \epsilon)/\gamma$.
Since all the requirements for (\ref{errorbound}) have now been satisfied,
there exists a unique
$\g{\theta} \in \C{B}_r(\m{0})$ such that
$\m{g}(\g{\theta}) = \m{0}$, or equivalently, such that
$\m{y}_{\theta E}(T) = \m{b}_E + \g{\pi}$.
By (\ref{errorbound}), $\|\g{\theta}\| \le c \|\g{\pi}\|$,
where $c = \gamma/(1-\epsilon \gamma)$ is independent of $\g{\pi}$.
Since $\m{y}_\theta$ satisfies both the initial and terminal conditions
for $\m{x}^\pi$ in (\ref{P}), we can take $\m{x}^\pi = \m{y}_{\theta}$.
At $t = 0$,
\[
\m{x}_J^\pi (0) = \m{y}_{\theta J} (0) = \g{\theta}^* + \g{\theta},
\]
which rearranges to give (\ref{xpi-stable}).
As $\epsilon$ tends to zero, we can let $r$ also approach zero, in which case
the denominator in (\ref{errorbound}) tends to one and the ball containing
the solution $\g{\theta}$ to $\m{g}(\g{\theta}) = \m{0}$ tends to zero.
\end{proof}
\section{Stability with respect to the Switch Point}
\label{SSP}
In order to obtain the derivative of the objective in (\ref{SP}) with respect
to the switch point, we need to analyze the effect of perturbations in
the switch point $s$.
Let $\m{F}^{+}$ be defined by
\[
\m{F}^{+} (\m{x}, t) = \left\{ \begin{array}{l}
\m{F}_0 (\m{x}, t) \mbox{ for all } t \in [0, s+\Delta s), \\
\m{F}_1 (\m{x}, t) \mbox{ for all } t \in (s+\Delta s, T],
\end{array}
\right.
\]
where $|\Delta s| \le \rho$.
Hence, $\m{F}^{+}$ is the dynamics gotten by changing the
switch point from $s$ to $s + \Delta s$.
The boundary-value problem associated with the perturbed switch point is
\begin{equation}\label{s+}
\m{\dot{x}}(t) =
\m{F}^+(\m{x}(t),t ), \quad \m{x}_I (0) = \m{b}_I, \quad \m{x}_E (T) = \m{b}_E,
\end{equation}
and a solution, if it exists, is denoted $\m{x}^+$.
The goal in this section is to show that when the invertibility condition of
Lemma~\ref{IC-stability} holds, the perturbed problem (\ref{s+}) has
a solution that is stable with respect to the perturbation $\Delta s$.

Let $\m{y}_\theta^{+}$ denote the solution to the perturbed
initial-value problem
\begin{equation}\label{IVP+}
\m{\dot{y}}(t) =
\m{F}^+(\m{y}(t),t ), \quad \m{y}_I (0) = \m{b}_I,
\quad \m{y}_J (0) = \g{\theta}^* + \g{\theta},
\end{equation}
where $\g{\theta}^* = \m{x}_J (0)$.
When $\g{\theta} = \m{0}$, we omit the $\g{\theta}$ subscript on
$\m{y}_\theta^{+}$ so $\m{y}^{+} := \m{y}_0^+$.
Since $\m{F}^{+} = \m{F}_0$ on $[0, s)$, assuming $\Delta s > 0$,
it follows that
\begin{equation}\label{I-pert}
\m{y}^{+} (t) = \m{y}_0^+ (t) = \m{x}(t) \; \mbox{ for all } t \in [0, s).
\end{equation}
For $t \in (s, T]$, it is shown in 
\cite[(2.12)--(2.14)]{AghaeeHager21} that
\begin{equation}\label{E-pert}
\|\m{y}^{+}(t) - \m{x}(t)\| = \C{O} (\Delta s) \mbox{ on }
(s, T], \mbox{ which implies }
\m{y}_E^+ (T) = \m{b}_E - \g{\pi}
\end{equation}
for some $\g{\pi} = \C{O}(\Delta s)$ since $\m{x}(T) = \m{b}_E$.
By (\ref{I-pert}) and (\ref{E-pert}),
$\m{y}^+$ lies inside the tubes around $\m{x}$
given in Dynamic Smoothness when $\Delta s$ is sufficiently small.
Moreover, as in (\ref{yx}), it follows from
Dynamics Smoothness and
\cite[Cor. 2.3]{AghaeeHager21} that (\ref{IVP+}) has a solution
$\m{y}_{\theta}^+$ when $|\Delta s| \le \rho$ and
$\|\g{\theta}\|$ is sufficiently small, and we have the bound
\begin{equation}\label{yx+}
\|\m{y}_{\theta}^+(t) - \m{y}^+(t)\| \le e^{Lt} \|\g{\theta} \|
\quad \mbox{for all } t \in [0, T].
\end{equation}
Combine (\ref{I-pert})--(\ref{yx+}), and the triangle inequality to obtain
\begin{equation}\label{ytheta+x}
\|\m{y}_{\theta}^+(t) - \m{x}(t)\| =
\C{O}(\Delta s) + \C{O}(\g{\theta}) \quad
\mbox{for all } t \in [0, T].
\end{equation}

Now let us consider whether a solution exists to (\ref{s+}),
assuming a solution to the original system (\ref{s}) exists when $\Delta s = 0$.
As in the previous section, our approach is to focus on the initial-value
problem (\ref{IVP+}) and try to choose $\g{\theta}$ such that
$\m{y}_{\theta}^+ = \m{x}^+$ is a solution of (\ref{s+}).
In particular, if we choose $\g{\theta}$ such that
\[
\left( \m{y}_{\theta}^+(T) - \m{y}_0^+ (T) \right)_E = \g{\pi},
\]
then combining this with (\ref{E-pert}) gives
\[
\m{y}_{\theta E}^+ = 
\m{y}_{E}^+ + \g{\pi} = \m{b}_E - \g{\pi} + \g{\pi} = \m{b}_E.
\]
Thus $\m{y}_\theta^+$ satisfies the same boundary conditions as
those for a solution $\m{x}^+$ of (\ref{s+}).
With this insight, the following result is established:
\begin{lemma}\label{s-stability}
If $\g{\Phi}_{EJ}(T)$ is invertible, then for $\Delta s$ in a neighborhood
of $0$, the problem $(\ref{s+})$, with perturbed switch point $s + \Delta s$,
has a solution $\m{x}^+$, and we have
\begin{equation}\label{xpi-stable+}
\|\m{x}_J^{+}(0) - \m{x}_J (0)\| =
\|\m{x}_J^{+}(0) - \g{\theta}^*\| \le c |\Delta s|
\mbox{ for all } \Delta s \mbox{ near } 0,
\end{equation}
where $c$ is a constant that is independent of $\Delta s$.
\end{lemma}
\begin{proof}
The lemma is stated in terms of the fundamental matrix $\g{\Phi}$
that arises in the unperturbed problem of Section~\ref{SBC}, and which
satisfies
\[
\g{\dot{\Phi}} (t) = \nabla_x F (\m{x}(t), t) \g{\Phi}(t), \quad
\g{\Phi}(0) = \m{I}.
\]
If the proof technique of Lemma~\ref{IC-stability} is applied to the problem
(\ref{s+}) with a perturbed switch point, then the associated fundamental
matrix is the solution of
\begin{equation}\label{Phi+}
\g{\dot{\Phi}}^+ (t) = \nabla_x F^+ (\m{y}^+(t), t) \g{\Phi}^+(t), \quad
\g{\Phi}^+(0) = \m{I}.
\end{equation}
Since $\m{x}(t) = \m{y}_0^+(t) = \m{y}^+(t)$ and $\m{F}^+ = \m{F}$
on the interval $[0, s]$, it follows that
$\g{\Phi}^+(t) = \g{\Phi} (t)$ on $[0, s]$.
On the interval $[s, s+\Delta s]$, $\g{\Phi}$ is associated with the dynamics
$\m{F}_1$ while $\g{\Phi}^+$ is associated with the dynamics  $\m{F}_0$,
so the fundamental matrices satisfy
\[
\g{\dot{\Phi}} (t) = \nabla_x \m{F}_1 (\m{x}(t), t) \g{\Phi}(t)
\quad \mbox{and} \quad
\g{\dot{\Phi}}^+ (t) = \nabla_x \m{F}_0 (\m{y}^+(t), t) \g{\Phi}^+(t) 
\quad \mbox{on } [s, s+\Delta s]
\]
with the initial condition $\g{\Phi} (s) = \g{\Phi}^+ (s)$.
Since $\m{F}_0$ and $\m{F}_1$ are smooth and the starting conditions for
$\g{\Phi} (t)$ and $\g{\Phi}^+ (t)$ at $t = s$ are the same,
it follows that the difference $\m{D} =\g{\Phi}^+ - \g{\Phi}$ satisfies
$\|\m{D}(s + \Delta s)\| = \C{O}(\Delta s)$.
On the interval $[s+\Delta s, T]$, the fundamental matrices satisfy
\[
\g{\dot{\Phi}} (t) = \nabla_x \m{F}_1 (\m{x}(t), t) \g{\Phi}(t)
\quad \mbox{and} \quad
\g{\dot{\Phi}}^+ (t) = \nabla_x \m{F}_1 (\m{y}^+(t), t) \g{\Phi}^+(t) .
\]
Subtracting the two equations, the difference $\m{D}$ satisfies
\begin{equation}\label{ats+}
\m{\dot{D}} (t) = \nabla_x \m{F}_1 (\m{y}^+(t), t) \m{D}(t)
+ [ \nabla_x \m{F}_1(\m{x}(t),t) - \nabla_x \m{F}_1(\m{y}^+(t),t)] \g{\Phi}(t),
\end{equation}
where $\m{D} (s + \Delta s) = \C{O}(\Delta s)$.
Choose $\Delta s$ small enough that $\m{y}^+$ lies within the tubes associated
with Dynamics Smoothness.
Hence, (\ref{E-pert}),
Dynamics Smoothness, and the Lipschitz property for $\nabla_x \m{F}_1$
imply that the coefficient of $\g{\Phi}$ in (\ref{ats+}) is $\C{O}(\Delta s)$.
By the boundedness of $\m{y}^+$ and $\g{\Phi}$, it follows that the
solution $\m{D}$ of the linear equation (\ref{ats+}) satisfies
$\m{D}(T) = \C{O}(\Delta s)$.
Since $\g{\Phi}_{EJ} (T)$ is invertible by assumption, then so is
$\g{\Phi}_{EJ}^+ (T)$ for $|\Delta s|$ sufficiently small and
$\g{\Phi}_{EJ}^+ (T)$ converges to $\g{\Phi}_{EJ} (T)$ as $\Delta s$ tends
to zero.
Let us take $\Delta s$ small enough that
$\|\g{\Phi}_{EJ}^{+}(T)^{-1}\| \le \gamma^+ := 2\gamma$.

Observe that the analysis of $\g{\Phi}$ and $\g{\Phi}^+$ concern the
case where $\g{\theta} = \m{0}$.
Next, $\g{\theta}$ is introduced into the analysis.
Similar to the approach in the proof of
Lemma~\ref{IC-stability}, we take
$\C{L} =$ $\nabla \m{g}(\m{0}) =$ $\g{\Phi}_{EJ}^+(T)$
where $\g{\Phi}^+$ is the solution of (\ref{Phi+}),
$\m{g}(\g{\theta}) =$
$\m{y}_{\theta E}^+(T) - \m{b}_E$, and $\m{y}^+_\theta$ is the
solution of (\ref{IVP+}).
Note that $\nabla \m{g}(\g{\theta})$ is the $EJ$ submatrix of
$\g{\Phi}_{\theta}^+ (T)$ where
\begin{equation}\label{Phi+theta}
\g{\dot{\Phi}}_{\theta}^+(t) =
\nabla_x \m{F}^+(\m{y}_{\theta}^+(t), t) \g{\Phi}_{\theta}^+ (t), \quad
\g{\Phi}^+ (0) = \m{I} .
\end{equation}
Subtract the equation (\ref{Phi+}) for $\g{\Phi}^+$
from (\ref{Phi+theta}) to obtain an equation for
the difference $\m{D}^+ = \g{\Phi}_{\theta}^+ - \g{\Phi}^+$:
\begin{equation}\label{D+DE}
\m{\dot{D}}^+(t) =
\nabla_x \m{F}^+(\m{y}_{\theta}^+ (t), t) \m{D}^+(t)
+ [ \nabla_x \m{F}^+ (\m{y}_{\theta}^+ (t), t) -
\nabla_x \m{F}^+ (\m{y}^+ (t), t)] \g{\Phi}^+(t) ,
\end{equation}
where $\m{D}^+ (0) = \m{0}$.
By the Lipschitz property for $\nabla_x \m{F}_0$ and $\nabla_x \m{F}_1$
and by (\ref{yx+}), the coefficient of $\g{\Phi}^+$ in
(\ref{D+DE}) is $\C{O}(\g{\theta})$ when $|\Delta s| \le \rho$ and
$\g{\theta}$ is sufficiently small.
Since $\m{y}^+_\theta$ and $\g{\Phi}^+$ are both uniformly bounded,
it follows from (\ref{D+DE}) that
$\|\m{D}^+(T)\| = \C{O}(\g{\theta})$.
In our context, the left side of (\ref{er}) is
\[
\|\nabla \m{g}(\g{\theta}) - \nabla \m{g}(\m{0})\| \le
\|\g{\Phi}_\theta^+ (T) - \g{\Phi}^+ (T)\| =
\|\m{D}^+(T)\| \le c\|\g{\theta}\|,
\]
for some constant $c$ independent of $\g{\theta}$ and $|\Delta s| \le \rho$.
Choose $\epsilon > 0$ such that $\epsilon \gamma^+ < 1$, and
choose $r$ small enough that $\|\m{D}^+(T)\| \le \epsilon$ when
$\|\g{\theta}\| \le r$.

By (\ref{E-pert}), $\delta =$ $\|\m{g}(\m{0})\| =$
$\|\m{y}_E^+(T) - \m{b}_E\| =$ $\C{O}(\Delta s)$.
Choose $\Delta s$ smaller, if necessary, to ensure that
$\delta \le r(1-\gamma^+ \epsilon)/\gamma^+$.
Hence, by Proposition~\ref{asen}, there exists a unique
$\g{\theta} \in \C{B}_r(\m{0})$ such that
$\m{g}(\g{\theta}) = \m{0}$, or equivalently, such that
$\m{y}_{\theta E}^+(T) = \m{b}_E$.
Moreover, $\m{x}^+ = \m{y}_{\theta}^+$ is a solution of the perturbed problem
(\ref{s+}) and $\|\g{\theta}\| \le c |\Delta s|$ where
$c = \gamma^+/(1-\epsilon \gamma^+)$ by (\ref{errorbound}).
The identity $\m{x}^+ = \m{y}_{\theta}^+$ implies that
\[
\m{x}^+_J (0) = \m{y}^+_{\theta J} (0) = \g{\theta}^* + \g{\theta},
\]
which rearranges to give (\ref{xpi-stable+}) since
$\g{\theta} = \C{O}(\Delta s)$.
\end{proof}

\section{Objective Derivative with Respect to Switch Point}
\label{DSP}
Lemmas~\ref{IC-stability} and \ref{s-stability} will be combined
to establish the formula (\ref{DC}) for the derivative of the objective with
respect to a switch point.
Notice that this formula involves the costate $\m{p}$, which must satisfy
complementary boundary conditions to those of $\m{x}$.
Since the costate equation is linear, its solution can be expressed in terms
of a fundamental matrix denoted $\g{\Psi}$, the unique solution of the
initial-value problem
\[
\g{\dot{\Psi}} = - \nabla_x \m{F}(\m{x}(t),t))\tr \g{\Psi}(t), \quad
\g{\Psi}(0) = \m{I}.
\]
Since $\m{p}_J (0) = \m{0}$ while $\m{p}_F (T) = \nabla_F C(\m{x}(T))$,
a solution to the costate equation exists when $\g{\Psi}_{FI}(T)$ is invertible.
\begin{theorem}\label{sp-theorem}
If Dynamics Smoothness holds, the objective $C$ is continuously
differentiable, and both $\g{\Phi}_{EJ}(T)$ and $\g{\Psi}_{FI}(T)$
are invertible, then
\begin{equation} \label{DCs}
\frac{\partial C}{\partial s} (s) =
H_{0} (\m{x}(s), \m{p}(s), s)
- H_{1}(\m{x}(s), \m{p}(s), s),
\end{equation}
where $H_j (\m{x}, \m{p}, t) = \m{p}\m{F}_j(\m{x}, t)$, $j = 0$ or $1$,
and the row vector $\m{p}: [0, T] \rightarrow \mathbb{R}^n$
is the solution to the linear differential equation
\begin{equation}\label{pd}
\m{\dot{p}}(t) = - \m{p}(t) \nabla_x \m{F}(\m{x}(t), t),
\quad t \in [0, T],
\quad
\m{p}_F(T) = \nabla_F C(\m{x}(T)), \quad \m{p}_J(0) = \m{0}.
\end{equation}
\end{theorem}

\begin{proof}
By Lemma~\ref{s-stability}, the problem with perturbed switch point has a
solution $\m{x}^+$ for $\Delta s$ sufficiently small.
Our goal is to evaluate the limit
\[
\lim_{\Delta s \rightarrow 0} \frac{C(\m{x}^+(T)) - C(\m{x}(T))} {\Delta s} .
\]

Let $\m{y}_{\theta}^+$ be the solution of (\ref{IVP+}) associated with the
solution $\m{x}^+$ of (\ref{s+}); that is,
$\m{y}_{\theta}^+ = \m{x}^+$.
Let $\m{Z}$ be the solution to the following linearized system:
\begin{eqnarray}
\m{\dot{Z}}(t) &=& \nabla_x \m{F}_0 (\m{x}(t), t) \m{Z}(t), \quad
t \in [0, s), \quad \m{Z}_I(0) = \m{0}, \quad \m{Z}_J (0) = \g{\theta},
\label{eq0}\\
\m{\dot{Z}}(t) &=& \nabla_x \m{F}_1 (\m{x}(t), t) \m{Z}(t), \quad
t \in (s+\Delta s, T], \label{eq1}
\end{eqnarray}
where
\begin{equation}\label{Z-update}
\m{Z}(s+\Delta s) =
\m{Z}(s) + \Delta s [\m{F}_0 (\m{x}(s), s) - \m{F}_1 (\m{x}(s), s)].
\end{equation}
There is a unique solution to (\ref{eq0})--(\ref{Z-update}) due to the linearity
of the first two equations.
Since $\g{\theta} = \C{O}(\Delta s)$ by Lemma~\ref{s-stability}
and the coefficient of $\m{Z}$ in (\ref{eq0}) is continuous, it follows that
$\m{Z}(t) =$ $\C{O}(\Delta s)$ for $t \in [0, s]$.
Since $\m{F}_0 (\m{x}(s), s)$ and $\m{F}_1 (\m{x}(s), s)$ are both continuous
for $t \in [s, s+\rho]$, $\|\m{Z}(s + \Delta s)\| = \C{O}(\Delta s)$.
Finally, due to the linearity of (\ref{eq1}), we have
\begin{equation}\label{Zbound}
\m{Z}(t) = \C{O}(\Delta s) \quad \mbox{for }
t\in [0, s] \cup [s+\Delta s, T].
\end{equation}

The difference between $\m{y}_{\theta}^+ - \m{x}$ and $\m{Z}$ can be
analyzed as in Section~\ref{SBC} in terms of
$\g{\delta}(t) = \m{y}_{\theta}^+(t) - \m{x}(t) - \m{Z}(t)$.
By the initial conditions for $\m{y}_{\theta}^+$, for $\m{x} =$ $\m{y}_0$,
and for $\m{Z}$ in (\ref{IVP+}), (\ref{IVP}), and (\ref{eq0}) respectively,
it follows that $\g{\delta}(0) = \m{0}$.
Exactly the same expansions between (\ref{deltadot}) and (\ref{deltabound})
yield $\|\g{\delta}(t)\| =$ $\C{O}(\|\g{\theta}\|^2)$ for all $t \in [0, s]$.
Moreover, from Lemma~\ref{s-stability} and the fact that $\g{\theta}$ is chosen
such that $\m{y}_{\theta}^+ = \m{x}^+$, we have
$\|\g{\theta}\| \le c |\Delta s|$.
Hence,
\begin{equation}\label{[0, s]}
\|\g{\delta}(t)\| = \C{O}(|\Delta s|^2) \quad \mbox{on } [0, s].
\end{equation}

Now consider the interval $[s, s+\Delta s]$, $|\Delta s| \le \rho$.
Since $\m{x}^+$ and $\m{x}$ are twice continuously differentiable on
$(s, s+\Delta s)$, a Taylor expansion gives
\begin{eqnarray}
\m{x}^+(s + \Delta s) &=& \m{x}^+(s) + \Delta s \m{F}_0 (\m{x}^+(s), s)
+ \C{O}(|\Delta s|^2), \label{first} \\
\m{x}(s + \Delta s) &=& \m{x}(s) + \Delta s \m{F}_1 (\m{x}(s), s)
+ \C{O}(|\Delta s|^2). \label{second}
\end{eqnarray}
Subtracting (\ref{second}) and (\ref{Z-update}) from
the (\ref{first}) and referring to the definition of $\g{\delta}$ yields
\begin{equation} \label{deltas+ds}
\g{\delta}(s+\Delta s) = \g{\delta}(s) +
\Delta s [\m{F}_0 (\m{x}^+(s),s) - \m{F}_0(\m{x}(s),s)] + \C{O}(|\Delta s|^2) .
\end{equation}
By (\ref{ytheta+x}) and the fact established in Lemma~\ref{s-stability}
that $\m{y}_{\theta}^+ = \m{x}^+$ with
$\g{\theta} = \C{O}(\Delta s)$, we have
$\|\m{x}^+(s) - \m{x}(s)\| = \C{O}(\Delta s)$.
Due to Dynamics Smoothness and
the Lipschitz continuity of $\m{F}_0$, and the fact from (\ref{[0, s]})
that $\g{\delta}(s) = \C{O}(|\Delta s|^2)$,
(\ref{deltas+ds}) implies that
$\g{\delta}(s+\Delta s) = \C{O}(|\Delta s|^2)$.

The final interval $[s + \Delta s, T]$ is treated exactly as in the expansions
(\ref{deltadot})--(\ref{deltabound}) except that
$\g{\delta}(0) = \m{0}$ in (\ref{deltabound}) should be replaced by
$\g{\delta}(s + \Delta s) = \C{O}(|\Delta s|^2)$.
Nonetheless, we have
$\|\g{\delta}(t)\| =$ $\C{O}(|\Delta s|^2)$ for all
$t \in [s+\Delta s, T]$.
In summary,
\begin{equation}\label{deltaE}
\|\g{\delta}(t)\| = \C{O}(|\Delta s|^2) \quad \mbox{for all }
t \in [0, s] \cup [s+\Delta s, T].
\end{equation}

If $\m{p}$ is the solution of (\ref{pd}), which exists by the invertibility
assumption for $\g{\Psi}_{FI} (T)$, and $\m{Z}$ is the solution
of (\ref{eq0})--(\ref{Z-update}), then we integrate over
$[s+\Delta s, T]$ and then integrate by parts to obtain
\begin{eqnarray}
0 &=& \int_{s+\Delta s}^T \m{p}(t)
\bigg[\nabla_x \m{F}(\m{x}(t), t)\m{Z}(t) - \m{\dot{Z}}(t)\bigg] \; dt
\nonumber \\
&=& \int_{s+\Delta s}^T \bigg[ \m{p}(t)
\nabla_x \m{F}(\m{x}(t), t) + \m{\dot{\m{p}}}(t)\bigg] \m{Z}(t) \; dt
- \m{p}(T)\m{Z}(T) +\m{p}(s+\Delta s) \m{Z}(s+\Delta s)
\nonumber  \\[.05in]
&=& -\m{p}_E (T)\m{Z}_E (T) -\m{p}_F (T)\m{Z}_F (T)
+\m{p}(s+\Delta s) \m{Z}(s+\Delta s)
\nonumber  \\[.05in]
&=& -\m{p}_E (T)\m{Z}_E (T) - \nabla_F C(\m{x}(T))\m{Z}_F(T)
\nonumber \\[.05in]
&& \quad + \m{p}(s+\Delta s) [\m{Z}(s) + \Delta s(\m{F}_0(\m{x}(s), s)
- \m{F}_1(\m{x}(s), s))] , \label{[s,T]}
\end{eqnarray}
where the integral in the second equality vanishes due to (\ref{pd}) and
the last equality is due to (\ref{Z-update}).
Similarly, an integral over $[0, s]$ yields
\begin{eqnarray}
0 &=& \int_{0}^s \m{p}(t)
\bigg[\nabla_x \m{F}(\m{x}(t), t)\m{Z}(t) - \m{\dot{Z}}(t)\bigg] \; dt
\nonumber \\
&=& \int_{0}^s \bigg[ \m{p}(t)
\nabla_x \m{F}(\m{x}(t), t) + \m{\dot{\m{p}}}(t)\bigg] \m{Z}(t) \; dt
- \m{p}(s)\m{Z}(s) +\m{p}(0) \m{Z}(0)
\nonumber  \\[.05in]
&=&  -\m{p}(s)\m{Z}(s)
\label{[0,s]}
\end{eqnarray}
since $\m{p}_J(0) = \m{0} = \m{Z}_I(0)$.

Since $C$ is continuously differentiable at
$\m{x}(T)$, the mean-value theorem gives
\begin{equation}\label{deltaC}
C(\m{x}^+(T)) - C(\m{x}(T)) =
\nabla_F C(\m{x}_{\Delta})[\m{x}_F^+ (T) - \m{x}_F (T)],
\end{equation}
where $\m{x}_\Delta$ is a point on the line segment connecting
$\m{x}^+(T)$ and $\m{x}(T)$.
Add (\ref{[s,T]})--(\ref{deltaC}) and substitute
\[
\m{x}^+(T) - \m{x}(T) =
\m{x}^+(T) - \m{x}(T) -\m{Z}(T) + \m{Z}(T) = \g{\delta}(T) + \m{Z}(T)
\]
to obtain
$C(\m{x}^+(T)) - C(\m{x}(T)) =$
\begin{eqnarray}
&\nabla_F C(\m{x}_{\Delta})\g{\delta}_F(T) +
[\nabla_F C(\m{x}_{\Delta})- \nabla_F C(\m{x}(T)]\m{Z}_F (T)
+ [\m{p}(s+\Delta s) - \m{p}(s)]\m{Z}(s) & \nonumber \\[.1in]
& -\m{p}_E(T) \m{Z}_E(T) + \Delta s \m{p}(s+\Delta s)
[\m{F}_0(\m{x}(s), s) - \m{F}_1(\m{x}(s), s)]. &
\label{final}
\end{eqnarray}

Bounds are now obtained for each of the terms in (\ref{final}).
By (\ref{deltaE}), $\|\g{\delta}(T)\| = \C{O}(|\Delta s|^2)$ so
$|\nabla_F C(\m{x}_{\Delta})\g{\delta}_F(T)| = \C{O}(|\Delta s|^2)$.
Since the distance between $\m{x} (t)$ and $\m{x}^+(t) = \m{y}_\theta^+ (t)$
is $\C{O}(\Delta s)$ by (\ref{ytheta+x}) and Lemma~\ref{s-stability},
the distance between $\m{x}_\Delta$ and $\m{x}(T)$ is also $\C{O}(\Delta s)$.
Since $\m{Z}(T) = \C{O}(\Delta s)$ by (\ref{Zbound}),
it follows that $\m{Z}_F (T) = \C{O}(\Delta s)$,
while the coefficient of $\m{Z}_F$ tends to zero as $\Delta s$ tends to zero.
Similarly, $\m{Z}(s) = \C{O}(\Delta s )$ by (\ref{Zbound}),
and the coefficient of $\m{Z}(s)$ tends to $0$ as $|\Delta s|$ tends to $0$.
Finally,
since $\m{x}_E^+(T) = \m{x}_E(T) = \m{b}_E$ and
$\g{\delta} (T) = \C{O}(|\Delta s|^2)$, it follows that
\[
\C{O}(|\Delta s|^2) = \|\g{\delta}_E (T)\| =
\|\m{x}_E^+(T) - \m{x}_E(T) - \m{Z}_E(T)\| = \|\m{Z}_E(T)\|,
\]
which implies that $\m{p}_E(T) \m{Z}_E(T) =\C{O}(|\Delta s|^2)$.
Divide (\ref{final}) by $\Delta s$ and let $\Delta s$ tend
to zero to obtain
\[
\frac{\partial C}{\partial s} (s) =
\lim_{\Delta s \rightarrow 0}
\frac{C(\m{x}^+(T)) - C(\m{x}(T))}{\Delta s}
=  \m{p}(s) [\m{F}_0(\m{x}(s), s) - \m{F}_1(\m{x}(s), s)] ,
\]
which completes the proof.
\end{proof}

\section{Singular Control Depending on Both State and Costate}
\label{case2}
The case where a singular control depends on both the state and costate was
analyzed in \cite[Sect.~3]{AghaeeHager21}.
The basic idea is to view the state/costate pair $(\m{x}, \m{p})$ as a new
generalized state variable that must satisfy the endpoint conditions
appearing in the first-order optimality conditions.
Next, a pair of generalized co-states are introduced corresponding to the
state and costate dynamics, which leads to a generalized Hamiltonian.
The formula for the derivative of the objective with respect to a switch point
is the same as the formula in the original formulation, however,
the Hamiltonian is replaced by the generalized Hamiltonian.
The reader is referred to \cite[Sect.~3]{AghaeeHager21} for details.

\section{Algorithms}
\label{numerical}
The derivative obtained in this paper is very useful when solving a
singular control problem using gradient-based methods;
however, a good starting guess for the switching points is needed.
One useful approach for generating an initial guess is to employ an Euler
discretization with total variation regularization, as explained in
\cite[Sect.~5]{AghaeeHager21} and with more detail in
\cite{Atkins23}.

When a problem has multiple switch points, the derivative with respect to all
the switch points can be computed with one integration of the state dynamics,
followed by one integration of the costate dynamics.
Since the derivative with respect a switch point is related to the
Hamiltonian change at the switch point, both the dynamics and the costate should
be evaluated accurately at the switch points.

When evaluating the objective or its gradient, one must also find the state
that satisfies the boundary conditions.
Similar to the analysis in the paper, the state that satisfies the boundary
conditions can be computed by choosing $\g{\theta}$ so that
$\m{y}_\theta = \m{x}$ satisfies the boundary conditions.
Newton's method is often a good approach for computing $\g{\theta}$.

\section{Conclusions}
\label{conclusions}
The Switch Point Algorithm of \cite{AghaeeHager21} for an initial-value problem
was extended to handle both initial and terminal boundary conditions.
The formula for the derivative of the objective with respect to a switch
point reduced to the change in the Hamiltonian across a switch point.
This was the same formula obtained for an initial-value problem.
Nonetheless, significant modifications in the analysis were needed to handle
terminal constraints.
In particular, the existence and stability of solutions to a boundary-value
problem under perturbations in the terminal constraint and in the
switch points needed to be analyzed, and the
invertibility of certain matrices connected with the
linearized state equation and with the costate equation were required.

\section{Acknowledgments}
Many thanks to Christian Austin for pointing out Taylor's book
\cite{TaylorPDE} which provides in Chapter~1.6 a compact treatment
of differentiability for the solution of a differential equation
with respect to an initial condition.
\bigskip

\noindent
{\bf Data Availability}
Data sharing is not applicable to this article as no datasets were
generated or analyzed during the current study.
\smallskip

\noindent
{\bf Conflict of Interest}
The author has no competing interests to declare that are relevant
to the content of this article.
\bibliographystyle{siam}

\begin{thebibliography}{10}

\bibitem{AghaeeHager21}
{\sc M.~Aghaee and W.~W. Hager}, {\em The switch point algorithm}, SIAM
  J.~Control~Optim., 59 (2021), pp.~2570--2593.

\bibitem{Aly1978}
{\sc G.~Aly}, {\em The computation of optimal singular control}, International
  Journal of Control, 28 (1978), pp.~681--688.

\bibitem{Anderson1972}
{\sc G.~M. Anderson}, {\em An indirect numerical method for the solution of a
  class of optimal control problems with singular arcs}, IEEE Trans. Automat.
  Control,  (1972), pp.~363--365.

\bibitem{Biegler20}
{\sc O.~Andr{\'e}s-Mart{\'i}nez, L.~T. Biegler, and A.~Flores-Tlacuahuac}, {\em
  An indirect approach for singular optimal control problems}, Computers and
  Chemical Engineering, 139 (2020), pp.~106923: 1--12.

\bibitem{Biegler19b}
{\sc O.~Andr{\'e}s-Mart{\'i}nez, A.~Flores-Tlacuahuac, S.~Kameswaran, and L.~T.
  Biegler}, {\em An efficient direct/indirect transcription approach for
  singular optimal control}, American Institute of Chemical Engineers Journal,
  65 (2019), pp.~937--946.

\bibitem{Atkins21}
{\sc S.~Atkins}, {\em Regularization of Singular Control Problems that Arise in
  Mathematical Biology}, PhD thesis, Department of Mathematics, University of
  Florida, Gainesville, FL, 2021.

\bibitem{Atkins23}
{\sc S.~Atkins, M.~Aghaee, M.~Martcheva, and W.~Hager}, {\em Solving singular
  control problems in mathematical biology using {PASA}}, in Computational and
  Mathematical Population Dynamics, N.~Tuncer, M.~Martcheva, O.~Prosper, and
  L.~Childs, eds., World Scientific, 2023, ch.~9, pp.~319--419.

\bibitem{Betts2010}
{\sc J.~T. Betts}, {\em Practical Methods for Optimal Control Using Nonlinear
  Programming, second edition}, SIAM, Philadelphia, 2010.

\bibitem{Bryson75}
{\sc A.~E. Bryson and Y.-C. Ho}, {\em Applied Optimal Control}, Hemisphere
  Publishing, New York, 1975.

\bibitem{Biegler16}
{\sc W.~Chen and L.~T. Biegler}, {\em Nested direct transcription optimization
  for singular optimal control problems}, American Institute of Chemical
  Engineers Journal, 62 (2016), pp.~3611--3627.

\bibitem{Biegler19a}
{\sc W.~Chen, Y.~Ren, G.~Zhang, and L.~T. Biegler}, {\em A simultaneous
  approach for singular optimal control based on partial moving grid}, American
  Institute of Chemical Engineers Journal, 65 (2019), pp.~e16584: 1--10.

\bibitem{DarbyHagerRao11}
{\sc C.~L. Darby, W.~W. Hager, and A.~V. Rao}, {\em Direct trajectory
  optimization using a variable low-order adaptive pseudospectral method},
  Journal of Spacecraft and Rockets, 48 (2011), pp.~433--445.

\bibitem{DarbyHagerRao10}
\leavevmode\vrule height 2pt depth -1.6pt width 23pt, {\em An hp-adaptive
  pseudospectral method for solving optimal control problems}, Optim. Control
  Appl. Methods, 32 (2011), pp.~476--502.

\bibitem{HagerDontchevPooreYang95}
{\sc A.~Dontchev, W.~W. Hager, A.~Poore, and B.~Yang}, {\em Optimality,
  stability, and convergence in nonlinear control}, Applied Math. and Optim.,
  31 (1995), pp.~297--326.

\bibitem{DontchevHagerVeliov00}
{\sc A.~L. Dontchev, W.~W. Hager, and V.~M. Veliov}, {\em Second-order
  {Runge}-{Kutta} approximations in constrained optimal control}, {SIAM} J.
  Numer. Anal., 38 (2000), pp.~202--226.

\bibitem{Hager90}
{\sc W.~W. Hager}, {\em Multiplier methods for nonlinear optimal control},
  {SIAM} J. Numer. Anal., 27 (1990), pp.~1061--1080.

\bibitem{Hager2000RungeKuttaMI}
{\sc W.~W. Hager}, {\em Runge-kutta methods in optimal control and the
  transformed adjoint system}, Numerische Mathematik, 87 (2000), pp.~247--282.

\bibitem{Hager02b}
{\sc W.~W. Hager}, {\em Numerical analysis in optimal control}, in
  International Series of Numerical Mathematics, K.-H. Hoffmann, I.~Lasiecka,
  G.~Leugering, J.~Sprekels, and F.~Tr\"{o}ltzsch, eds., vol.~139,
  Basel/Switzerland, 2001, Birkhauser Verlag, pp.~83--93.

\bibitem{HagerHouRao19}
{\sc W.~W. Hager, H.~Hou, S.~Mohapatra, A.~V. Rao, and X.-S. Wang}, {\em
  Convergence rate for a {Radau} {\it hp}-collocation method applied to
  constrained optimal control}, Comput. Optim. Appl., 74 (2019), pp.~274--314.

\bibitem{bang_bang}
{\sc W.~W. Hager and R.~Rostamian}, {\em Optimal coatings, bang-bang controls,
  and gradient techniques}, Optim. Control Appl. Methods, 8 (1987), pp.~1--20.

\bibitem{Jacobson70}
{\sc D.~H. Jacobon, S.~B. Gershwin, and M.~M. Lele}, {\em Computation of
  optimal singular controls}, IEEE Trans. Automat. Control, 15 (1970),
  pp.~67--73.

\bibitem{LiuHagerRao15}
{\sc F.~Liu, W.~W. Hager, and A.~V. Rao}, {\em Adaptive mesh refinement method
  for optimal control using nonsmoothness detection and mesh size reduction},
  J. Franklin Inst., 352 (2015), pp.~4081--4106.

\bibitem{LiuHagerRao18}
\leavevmode\vrule height 2pt depth -1.6pt width 23pt, {\em Adaptive mesh
  refinement method for optimal control using decay rates of {Legendre}
  polynomial coefficients}, IEEE Trans. Control Sys. Tech., 26 (2018),
  pp.~1475--1483.

\bibitem{Maurer1976}
{\sc H.~Maurer}, {\em Numerical solution of singular control problems using
  multiple shooting techniques}, J. Optim. Theory Appl., 18 (1976),
  pp.~235--257.

\bibitem{Maurer05}
{\sc H.~Maurer, C.~B\"uskens, J.-H.~R. Kim, and C.~Y. Kaya}, {\em Optimization
  methods for the verification of second order sufficient conditions for
  bang--bang controls}, Optim. Control Appl. Methods, 26 (2005), pp.~129--156.

\bibitem{Neubert03}
{\sc M.~Neubert}, {\em Marine reserves and optimal harvesting}, Ecology
  Letters, 6 (2003), pp.~843--849.

\bibitem{PagerRao2022}
{\sc E.~R. Pager and A.~V. Rao}, {\em Method for solving bang-bang and singular
  optimal control problems using adaptive {Radau} collocation}, Comput. Optim.
  Appl., 81 (2022), pp.~857--887.

\bibitem{PattersonHagerRao15}
{\sc M.~A. Patterson, W.~W. Hager, and A.~V. Rao}, {\em A $ph$ mesh refinement
  method for optimal control}, Optimal Control Applications and Methods, 36
  (2015), pp.~398--421.

\bibitem{Schattler12}
{\sc H.~Sch\"attler and U.~Ledzewicz}, {\em Geometric Optimal Control},
  Springer, New York, 2012.

\bibitem{TaylorPDE}
{\sc M.~E. Taylor}, {\em Partial Differential Equations I. Basic Theory},
  Springer, New York, 2011.

\bibitem{VossenThesis}
{\sc G.~Vossen}, {\em Numerische L\"osungsmethoden, hinreichende
  Optimalit\"atsbedingungen und Sensitivit\"at-sanalyse f\"ur optimale
  bang-bang und singul\"are Steuerungen}, PhD thesis, Universit\"at M\"unster,
  Germany, 2006.

\bibitem{Vossen2010}
{\sc G.~Vossen}, {\em Switching time optimization for bang-bang and singular
  controls}, J. Optim. Theory Appl., 144 (2010), pp.~409--429.

\end{thebibliography}

\end{document}